\documentclass[11pt]{article}

\usepackage{amsmath,amssymb,amsfonts,amsthm,mathtools}
\usepackage{setspace,graphicx,enumitem,soul,xcolor,multirow,url}
\usepackage{placeins,framed,oubraces,booktabs}
\usepackage{hyperref}
\usepackage{tikz-cd}

\def\bs{\boldsymbol}

\newtheorem{lemma}{Lemma}

\usepackage{authblk}
\author[1,2,3]{\emph{Li-Chun Zhang}}
\affil[1]{\emph{Statistisk sentralbyraa, Norway}}
\affil[2]{\emph{University of Southampton, UK (L.Zhang@soton.ac.uk)}}
\affil[3]{\emph{Universitetet i Oslo, Norway}}
\title{Graph sampling by lagged random walk}
\date{}
  
\begin{document}

\maketitle

\begin{abstract} We propose a family of lagged random walk sampling methods in simple undirected graphs, where transition to the next state (i.e. node) depends on both the current and previous states---hence, lagged. The existing random walk sampling methods can be incorporated as special cases. We develop a novel approach to estimation based on lagged random walks at equilibrium, where the target parameter can be any function of values associated with finite-order subgraphs, such as edge, triangle, 4-cycle and others.
\end{abstract}

\noindent \emph{Key words:} random jump, stationary distribution, non-Markovian process, generalised ratio estimator, capture-recapture estimator

\section{Introduction}

Let $G = (U, A)$ be a simple undirected graph, with known node set $U$ and \emph{order} $N = |U|$, but unknown edge set $A$ or \emph{size} $R = |A|$. Discrete-time random walk (RW) in $G$ can only move from the current state $X_t =i$ to the next $X_{t+1} =j$ if $(ij)\in A$, where $i, j\in U$, such that the walk is confined within the component of $G$ where the initial state $X_0$ is located, given any $X_0 \in U$. For a \emph{targeted random walk (TRW,} Thompsn, 2006), the transition probabilities $\Pr(X_{t+1} =j | X_t = i)$ at each time step $t$ can be subject to other devices, such as random jumps or an acceptance-rejection mechanism for the proposed moves. The TRW stationary distribution has many applications to large often dynamic networks, such as PageRank (Brin and Page, 1998); see Masuda et al. (2017) and Newman (2010) for reviews. 

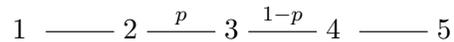
\begin{figure}[ht] 
\centering
\begin{tikzcd}[cramped]
1 \arrow[r, dash] & 2 \arrow[r, dash, "p"] & 3 \arrow[r, dash, "1-p"] & 4 \arrow[r, dash] & 5  
\end{tikzcd} 
\caption{Lagged random walk on a subway line, $(X_{t-1}, X_t) = (2,3)$} \label{fig:path}
\end{figure}

We call a random walk \emph{lagged} if the transition probabilities at time step $t$ are instead given by $\Pr(X_{t+1} =j | X_t = h, X_{t-1} =i)$. The process $\{ X_t : t\geq 0\}$ is non-Markovian when the states are generated by a \emph{lagged random walk (LRW)}. Figure \ref{fig:path} provides a simple setting of LRW on a subway line of 5 stations. Fixing $p=0.5$ for $X_t \in \{ 2, 3, 4\}$ yields pure RW, whereas normal subway service corresponds to
\[
p = \Pr(X_{t+1} = i \mid X_t =h, X_{t-1}=i) = 0
\]
for any $h \in \{ 2, 3, 4\}$. The LRW has a tendency to persist in the same travel direction if $p <0.5$, to backtrack to the previous station if $0.5 <p <1$, or stucks between two stations if $p=1$. 

Generally, by an extra parameter that controls the tendency for backtracking in arbitrary graphs, LRW yields a family of sampling methods which includes TRW as a special case. 

We are interested in the estimation of graph parameters as defined in Zhang and Patone (2017) and Zhang (2021). Let $G(M)$ be a subgraph induced by $M$, with node set $M$, $M\subset U$, and edge set $A\cap \{ (ij) : i,j\in M\}$. The specific characteristics of $G(M)$ is called \emph{motif}, denoted by $[M]$, and the order of $[M]$ is $|M|$, i.e. that of $G(M)$. Figure \ref{fig:motif} illustrates some low-order motifs. For instance, edge is a 2nd-order motif which can be given by $[\{ i,j\} : a_{ij} =1]$, where $a_{ij} =1$ if $(ij)\in A$ and 0 otherwise, for $i,j\in U$; any particular edge in $A$ is an \emph{occurrence} of the edge motif in $G$. Or, triangle is a 3rd-order motif given by $[\{ i,j,k\} : a_{ij} a_{jk} a_{ki} =1]$, where $i,j,k\in U$; any particular triangle in $G$ is an occurrence of the triangle motif. 

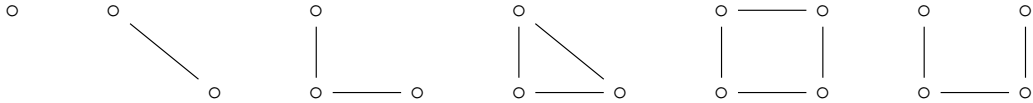
\begin{figure}[ht]
\centering
\begin{tikzcd}[cramped]
\circ & \circ \arrow[dr, dash] & & \circ \arrow[d, dash] & & \circ \arrow[d, dash] \arrow[dr, dash] &  
& \circ \arrow[d, dash] \arrow[r, dash] & \circ & \circ \arrow[d, dash] & \circ \\
& & \circ& \circ \arrow[r, dash] & \circ & \circ \arrow[r, dash] & \circ 
& \circ & \circ \arrow[l, dash] \arrow[u, dash] & \circ & \circ \arrow[l, dash] \arrow[u, dash]
\end{tikzcd} 
\caption{From left to right, motif node, edge, 2-star, triangle, 4-cycle and 3-path} \label{fig:motif}
\end{figure} 

Let $\Omega$ contain all the occurrences of the motifs of interest in $G$. For each $\kappa \in \Omega$, let $y_{\kappa}$ be a fixed value associated with $\kappa$. The corresponding \emph{graph total}, denoted by $\theta$, is given by
\begin{equation} \label{theta}
\theta = \sum_{\kappa \in \Omega} y_{\kappa}
\end{equation}
whereas we refer to any function of $\{ y_{\kappa} : \kappa \in \Omega\}$ as a \emph{graph parameter}, denoted by $\mu$. 

For instance, let $\Omega = U$, and let $y_{\kappa} = 1$ or 0 be a binary indicator of the type of node $\kappa$, then $\theta$ by \eqref{theta} is the total number of type-1 nodes in $G$, and $\mu = \theta/N$ is a 1st-order graph parameter measuring their prevalence. 
Let $\Omega = A$ and $y_{\kappa} =1$ for any $\kappa \in A$, then $\theta =R$ is the size of $G$. Let $\Omega$ contain all the triangles and 2-stars in $G$, and let $\theta$ be the graph total of triangles and $\theta'$ that of 2-stars, then $\mu = \theta/(\theta + \theta')$ is a measure of the transitivity of $G$. 

Previously, Thompson (2006) and Avrachenkov et al. (2010) have considered the estimation of 1st-order graph parameters based on TRW sampling. Below, in Section \ref{sampling}, we propose a more general family of LRW sampling methods. We develop an approach to estimating finite-order graph parameters and totals as defined above in Section \ref{estimation}. Numerical illustrations are given in Section \ref{illustration}, and some remarks on future research in Section \ref{remark}.

\section{LRW sampling} \label{sampling}

Let $d_i$ be the degree of node $i$ in $U$. Let $r>0$ and $0\leq w \leq 1$ be two chosen constants regulating random jump and backtracking, respectively. Let the transition probability of LRW be 
\begin{align}
& \Pr(X_t =j \mid X_t =h, X_{t-1} =i) \notag \\
& \hspace{10mm} = \begin{cases} \frac{r}{d_h+r} \big( \frac{1}{N} \big) + \frac{a_{hj}}{d_h +r} & \text{if } d_h = 1 \\
\frac{r}{d_h+r} \big( \frac{1}{N} \big) + \mathbb{I}(j=i) \frac{w a_{hj}}{d_h +r} 
+ \mathbb{I}(j\neq i) \frac{d_h - w a_{ih}}{d_h +r} \frac{a_{hj}}{d_h - a_{ih}} & \text{if } d_h > 1 
\end{cases} \label{LRW}
\end{align}
That is, the walk either jumps randomly to any node $j$ in $U$--including the current $X_t$--with probability $r/(d_h+r)$, or it moves to an adjacent node $j$ with probability $d_h/(d_h+r)$. If $d_h =1$, then there is only one adjacent $j$; if $d_h >1$, then it either backtracks to the previous $X_{t-1} =i$ with probability $w/(d_h+r)$, or it moves randomly to one of the other adjacent nodes.

\subsection{Stationary distribution}

TRW of Avrachenkov et al. (2010) is the special case of LRW by \eqref{LRW} given $w=1$, where the previous state $X_{t-1}$ (if adjacent) is treated as any other adjacent nodes at step $t$ and $\{ X_t : t\geq 0\}$ form a Markov chain. The stationary probability is given by 
\begin{equation} \label{pih}
\pi_h \coloneqq \Pr(X_t = h) = \frac{d_h +r}{2R + rN}
\end{equation}

The process $\{ X_t : t\geq 0\}$ is non-Markovian if $w<1$. However, let $\bs{x}_t = (X_{t-1}, X_t)$ for $t\geq 1$. Given any initial $\bs{x}_1 = (X_0, X_1)$, LRW \eqref{LRW} generates a Markov chain $\{ \bs{x}_t : t\geq 1\}$, since
\[
\Pr(\bs{x}_{t+1} | \bs{x}_t, ..., \bs{x}_1) = \Pr(\bs{x}_{t+1} | \bs{x}_t) = \Pr(X_{t+1} | \bs{x}_t)
\] 
Due to random jumps, this Markov chain is irreducible, such that there exists a unique stationary distribution, denoted by
\[
\pi_{\bs{x}} \coloneqq \Pr(\bs{x}_t = \bs{x}) \qquad\text{and}\qquad \sum_{\bs{x}} \pi_{\bs{x}} = 1
\]
A unique stationary distribution of $X_t$ follows, which satisfies the \emph{mixed} equation
\begin{equation} \label{mix}
\pi_h = \sum_{i\in U}  \Pr\big( \bs{x}_t = (i,h)\big) 
= \sum_{\substack{\bs{x} = (i,h)\\ i\in \nu_h}} \pi_{\bs{x}} + \sum_{i\not \in \nu_h}  \frac{\pi_i}{d_i +r} \big( \frac{r}{N} \big)
\end{equation}
where $\nu_h = \{ i\in U: a_{ih} =1\}$ is its neighbourhood of adjacent nodes, and the transition to $h$ from any node outside $\nu_h$ can only be accomplished by a random jump. 

\begin{lemma} 
For LRW by \eqref{LRW} with any $0\leq w\leq 1$ in undirected simple graphs, the stationary probability $\pi_h$ is given by \eqref{pih}
for any $h\in U$, and $\pi_{\bs{x}}$ for any $\bs{x} = (i,h)$ is given by
\[
\pi_{\bs{x}} \propto \begin{cases} 1 + r/N & \text{if } i\in \nu_h \\ r/N & \text{if } i\not \in \nu_h \end{cases}
\]
\end{lemma}

\begin{proof} Under LRW \eqref{LRW}, a flow between $\bs{x}_t = (i,h)$ and $\bs{x}_{t+1} = (h,j)$ in either direction is a flow over $(X_{t-1}, X_t, X_{t+1})$ in the same direction. To show the values $\{ p_h = d_h+r : h\in U\}$ satisfy the balanced flows through $(X_{t-1}, X_t, X_{t+1})$ at equilibrium, one needs to consider three situations. 
\begin{center}
\begin{tikzcd}[cramped]
\text{I.} & (i = j) \arrow[r,dash] & h &\hspace{20mm} i \arrow[r,dash] & h \arrow[r,dash] & j  
\end{tikzcd}
\end{center}
I. $\{ i, j\}\in \nu_h$, as illustrated. If $i= j$, then $(i,h,j)$ and $(j,h,i)$ are backtracking in either direction, the probability of which is the same by \eqref{LRW}, so that they are always balanced. If $i\neq j$, then neither $(i,h,j)$ nor $(j,h,i)$ is backtracking, with the same transition probability but $\pi_{\bs{x}_t}$ differ to $\pi_{\bs{x}_{t+1}}$ generally. However, $\bs{x} = (i,h)$ for any given $i\in \nu_h$ is involved in one flow in \emph{either} direction, such that \emph{altogether} these flows are balanced.  
\begin{center}
\begin{tikzcd}[cramped]
\text{II.} & i & h & j 
\end{tikzcd}
\end{center}
II. $\{ i, j\}\not \in \nu_h$, as illustrated, including $i=h$ or $j=h$, in which case $(i,h,j)$ and $(j,h,i)$ can only take place by random jumps, which are balanced given $p_h = d_h +r$, $\forall h\in U$, since
\[
p_i \frac{r}{d_i +r} \big( \frac{1}{N} \big) \frac{r}{d_h +r} \big( \frac{1}{N} \big) = 
p_j \frac{r}{d_j +r} \big( \frac{1}{N} \big) \frac{r}{d_h +r} \big( \frac{1}{N} \big) 
\] 
\begin{center}
\begin{tikzcd}[cramped]
\text{III.} & i \arrow[r,dash] & h & j &\hspace{20mm} i \arrow[r,dash] & (h = j)  
\end{tikzcd}
\end{center}
III. One of $(i,j)$, say, $i$ belongs to $\nu_h$ but not the other, including when $j=h$, as illustrated.
Setting $p_i = d_i +r$ for $\pi_i$ in the second term on the right-hand side of \eqref{mix}, we have 
\[
\big( \pi_h -  \sum_{\substack{\bs{x} = (i,h)\\ i\in \nu_h}} \pi_{\bs{x}} \big) \big( \frac{d_h}{d_h+r} + \frac{r}{d_h +r} \frac{d_h}{N} \big) 
= r \frac{N -d_h}{N} \big( \frac{d_h}{d_h+r} + \frac{r}{d_h +r} \frac{d_h}{N} \big) 
\]
i.e. summing over \emph{all} possible flows $(j,h,i)$ where $j\not \in \nu_h$ and $i\in \nu_h$. Whereas, summing over all the possible flows $(i,h,j)$ in the other direction, we have 
\[
\big( \sum_{\substack{\bs{x} = (i,h)\\ i\in \nu_h}} \pi_{\bs{x}} \big) \frac{r}{d_h +r} \big( \frac{N-d_h}{N} \big) 
\]
Balancing (i.e. equating) the two groups of flows at equilibrium yields
\[
\sum_{\substack{\bs{x} = (i,h)\\ i\in \nu_h}} \pi_{\bs{x}} = d_h + r \frac{d_h}{N} 
\] 
Replacing this into the right-hand side of \eqref{mix}, we obtain $p_h = d_h + r$. Summarising all the three situations, the values $p_h$ satisfy the balanced flows, so that $\pi_h \propto d_h +r$ as in \eqref{pih}. Finally, $\pi_{\bs{x}}$ for any $\bs{x} = (i,h)$ follows from \eqref{mix} by symmetry.  
\end{proof}

\subsection{Sample graph by LRW}

Zhang (2021) provides a general definition of sample graph as a subgraph of $G$. We apply it to LRW sampling as follows. By the \emph{observation procedure (OP)} of LRW, we observe all the edges incident to $X_t =i$ at step $t$, i.e. the entire row and column of the adjacency matrix of $G$ corresponding to node $i$. Let $s = \{ X_0, ..., X_T\}$ be the \emph{seed sample} of LRW, i.e. the nodes to which the OP is applied. For any given $s$, the observed part of the adjacency matrix can be specified as $s\times U \cup U\times s$, such that the sampled (or observed) edges are given by
\[
A_s = A \cap \big( s\times U \cup U\times s \big) 
\]
The \emph{sample graph} by LRW, denoted by $G_s$, is given by
\begin{equation} \label{Gs}
G_s = (U_s, A_s) \qquad\text{where}\qquad U_s = s \cup \mbox{Inc}(A_s) 
\end{equation}
i.e. the node sample is the union of $s$ and all the nodes incident to the edges in $A_s$.

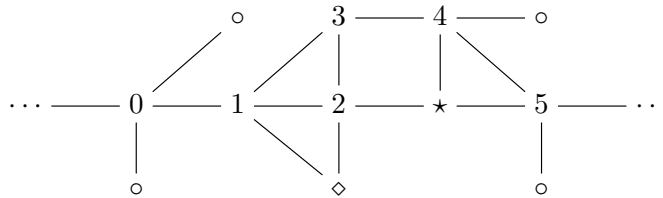
\begin{figure}[ht]
\centering 
\begin{tikzcd}[cramped]
& & \circ & 3 \arrow[r, dash] & 4 \arrow[dr, dash] \arrow[r, dash] \arrow[d, dash] & \circ \\
\cdots \arrow[r, dash] & 0 \arrow[r, dash] \arrow[d, dash] \arrow[ur, dash] 
& 1 \arrow[ur, dash] \arrow[r, dash] \arrow[dr, dash]
& 2 \arrow[u, dash] \arrow[r, dash] \arrow[d, dash] & \star \arrow[r, dash] & 5 \arrow[r, dash] \arrow[d, dash] & \cdots \\
& \circ & & \diamond & & \circ &
\end{tikzcd}
\caption{Illustration of sample observation by LRW} \label{fig:LRW} 
\end{figure} 

Figure \ref{fig:LRW} provides an illustration of sample observation by LRW. Suppose $X_{t+q} = q$ for $q=0, 1, ..., 5$. In addition to the nodes $\{ X_t, ..., X_{t+5}\}$ we observe many other motifs, such as edge $[\{ 0,1\}]$ at $X_t =0$ and again at $X_{t+1} =1$, 
triangle $[\{ 1,2,3\}]$ and $[\{ 1,2,\diamond \}]$ given $(X_{t+1}, X_{t+2}) = (1,2)$ and $[\{ 1, 2,3 \}]$ again given $(X_{t+2}, X_{t+3}) = (2,3)$,  4-cycle $[\{ 2,3,4,\star \}]$ given $(X_{t+2}, X_{t+3}, X_{t+4}) = (2,3,4)$. All these motifs are observed in the sample graph $G_s$ by LRW. Denote by $\Omega_s$ the elements of $\Omega$ in \eqref{theta} which are observed in $G_s$.

\section{Estimation} \label{estimation}

We consider first the estimation of $R$, i.e. the size of $G$, hence the stationary probability \eqref{pih} by LRW sampling. Next, we consider other finite-order graph parameters and totals. In applications one can use the mean of a given estimator from any number, say, $K$ independent LRWs as the final estimator, and use the empirical variance of the $K$ estimators for variance estimation.

\subsection{Size of graph}

Let $\{ X_{t_1}, ..., X_{t_{n_x}}\}$ be an extraction of $n_x$ states from LRW $\{ X_t : t \geq 0 \}$ at equilibrium, where $t_1, ..., t_{n_x}$ do not need to be consecutive. Similarly, let $\{ Y_{\tau_1}, ..., X_{\tau_{n_y}}\}$ be those of a separate independent LRW $\{ Y_{\tau} : \tau \geq 0 \}$. We have
\[
E(m) \coloneqq E\Big( \sum_{i=1}^{n_x} \sum_{j=1}^{n_y} \frac{1}{d_h +r} \mathbb{I}(Y_{\tau_j} = h \mid X_{t_i} = h) \Big) 
= \frac{n_x n_y}{2R + rN} 
\]
where $m$ can be referred to as the number of collisions by $\{ X_t \}$ and $\{ Y_{\tau} \}$. This yields 
\begin{equation} \label{CR}
\hat{R}_1 = (n_x n_y/m - rN)/2
\end{equation}
as a \emph{capture-recapture (CR)} estimator. Meanwhile, since
\[
E(\bar{d}_w) \coloneqq E\left( \frac{\sum_{i=1}^{n_x} \frac{d_{X_{t_i}}}{d_{X_{t_i}} +r} + \sum_{j=1}^{n_y} \frac{d_{Y_{\tau_j}}}{d_{Y_{\tau_j}} +r}}{\sum_{i=1}^{n_x} \frac{1}{d_{X_{t_i}} +r} + \sum_{j=1}^{n_y} \frac{1}{d_{Y_{\tau_j}} +r}} \right) \approx \frac{2R}{N}
\]
we obtain another \emph{generalised ratio (GR)} estimator of $R$ as
\begin{equation} \label{GR}
\hat{R}_2 = N \bar{d}_w/2
\end{equation}
Finally, combining the two, we obtain a GR-CR estimator, denoted by $\hat{R}_3$, by incorporating an estimator $\hat{N}$ although $N$ is known, i.e.
\begin{equation} \label{Hajek}
\begin{cases} 2m\hat{R}_3 + rm\hat{N} = n_x n_y \\ -2\hat{R}_3 + \bar{d}_w \hat{N} =0 \end{cases}
\Rightarrow \quad \hat{N} = \frac{n_x n_y}{m(r + \bar{d}_w)}
\quad\text{and}\quad \hat{R}_3 = \frac{n_x n_y \bar{d}_w}{2m(r + \bar{d}_w)} 
\end{equation}

\subsection{Finite-order graph parameters and totals}

Below we first clarify generally the basis of estimation under LRW sampling, before we define the estimators of finite-order graph parameters and totals.

\subsubsection{Basis of estimation}

Let $M = \{ X_{t+1}, ..., X_{t+q}\}$ be a set of successive states by LRW. At equilibrium, we have
\begin{equation} \label{piM}
\pi_M = \mbox{Pr}(X_{t+1}, ..., X_{t+q}) = \pi_{X_{t+1}} \prod_{i=1}^{q-1} p_{X_{t+i} X_{t+i+1}}
\end{equation}
where $p_{X_{t+i} X_{t+i+1}}$ is the transition probability from $X_{t+i}$ to $X_{t+i+1}$, which is known regardless the unvisited nodes of $G$. We shall refer to $\pi_M$ by \eqref{piM} as the \emph{stationary successive sampling probability (S3P)}, where the probability $\pi_{X_{t+1}}$ is known exactly up to the proportionality constant $2R+rN$, or it can be estimated via $\hat{R}$ in which case we write $\hat{\pi}_{X_{t+1}}$ and $\hat{\pi}_M$.  

Apart from the actual sequences in a LRW, the S3P \eqref{piM} is also available for any \emph{hypothetical} sequence $M$, given $M\subseteq s$, because the part of the transition probability matrix corresponding to $s\times s$ is already observed. For instance, given the actual sequence $(X_t, X_{t+1}, X_{t+2}) = (1, 2, 3)$ in Figure \ref{fig:LRW}, we can as well calculate $\pi_3 p_{32} p_{21}$ of a sampling sequence $(X_t, X_{t+1}, X_{t+2}) = (3, 2, 1)$ hypothetically. Given the seed sample $s$, all the sequences for which the S3P is available belong to the \emph{generating (sequences of) states} of LRW sampling given by
\[
\mathcal{C}_s = \{ M : M\subseteq s \}
\]
The subset containing the parts of the actual walk is given by 
\[
\mathcal{C}_w = \big\{ \{ X_t, ..., X_{t+q} \} : 0 \leq t \leq t+q \leq T \big\}
\]
For example, suppose LRW with $X_0 = 1$ and $X_T= 4$ in Figure \ref{fig:LRW}, for which we have
\[
\mathcal{C}_w = \big\{ \{ 1\}, \{ 1, 2\}, \{ 1,2,3\}, \{ 1,2,3, 4\}, \{ 2\}, \{ 2,3\}, \{ 2,3,4\}, \{ 3\}, \{ 3,4\}, \{ 4\} \big\} 
\]

Let the sample motif $\kappa$, $\kappa \in \Omega_s$, be observed from the \emph{actual sampling sequence of states (AS3)} $s_{\kappa} = (X_t, ..., X_{t+q})$, for some time step $t$ and $q = |s_{\kappa}| -1$. An \emph{equivalent sampling sequence of states (ES3)}, denoted by $\tilde{s}_{\kappa} \sim s_{\kappa}$ including $s_{\kappa}  \sim s_{\kappa}$, is any possible sequence of states of the same length $|\tilde{s}_{\kappa}| = |s_{\kappa}|$, such that the motif $\kappa$ would be observed if $(X_t, X_{t+1}, ..., X_{t+q}) = \tilde{s}_{\kappa}$ but not based on any subsequence of $\tilde{s}_{\kappa}$. In other words, a motif $\kappa$ observed from AS3 $s_{\kappa}$ can be sampled otherwise sequence-by-sequence, in which respect its ES3 set
\[
F_{\kappa} = \{ \tilde{s}_{\kappa} : \tilde{s}_{\kappa} \sim s_{\kappa} \}
\]
constitute the multiplicity of access to $\kappa$ under LRW sampling, similarly to the concept of multiplicity under indirect sampling (Birnbaum and Sirken, 1965). 

Take the triangle $\kappa$ with $M = \{ 1, 2, 3\}$ in Figure \ref{fig:LRW}, where $s_{\kappa} = (X_t, X_{t+1}) = (1,2)$ is an AS3, for which we have $F_{\kappa} = \{ (1,2)$, (2,1), (1,3), (3,1), (2,3), (3,2)$\}$, since the same $\kappa$ would be observed if $(X_t, X_{t+1})$ is any of these 6 sequences in $F_{\kappa}$.

\subsubsection{Estimators}
  
Given LRW sampling from $G$, choose a seed sample of successive states at equilibrium which is of size $n$, denoted by $s = \{ X_1, ..., X_n\}$. Obtain the corresponding motif sample $\Omega_s$, the generating states $\mathcal{C}_s$ and its subset $\mathcal{C}_w$.  

For any $\kappa \in \Omega_s$ observed from AS3 $s_{\kappa} = (X_t, ..., X_{t+q})$, where $s_{\kappa} \in \mathcal{C}_w$, let $F_{\kappa}$ be its ES3 set. Let $M$ be a possible sequence of states $(X_t, ..., X_{t+q})$. Let the sampling indicator $\delta_M =1$ if $M$ is realised and 0 otherwise, with the associated S3P $\pi_M$, over hypothetically repeated sampling. Let the observation indicator $I_{\kappa}(M) = 1$ if $M\in F_{\kappa}$ and 0 otherwise. Let $w_{M\kappa}$ be the \emph{incidence weight} for any $M \in F_{\kappa}$, where $w_{M\kappa} =0$ if $M \not \in F_{\kappa}$, such that for any $\kappa \in \Omega$, we have
\[
\sum_{M\in F_{\kappa}} w_{M\kappa} = \sum_{M} I_{\kappa}(M) w_{M\kappa} = 1 
\] 
An estimator of the graph total \eqref{theta} is given by
\begin{equation} \label{LRW-IWE}
\hat{\theta}(X_t, ..., X_{t+q}) = \sum_{\kappa \in \Omega} \sum_{M} \frac{\delta_M}{\pi_M} I_{\kappa}(M) w_{M\kappa} y_{\kappa}
\end{equation}
which is unbiased for $\theta$, since 
\[
E\big( \hat{\theta} (X_t, ..., X_{t+q}) \big) = \sum_{\kappa \in \Omega} y_{\kappa} \sum_{M} I_{\kappa}(M) w_{M\kappa} 
= \sum_{\kappa \in \Omega} y_{\kappa}  
\]

Although there are infinitely many possibilities in principle, the two basic choices of $w_{M\kappa}$ are the \emph{multiplicity weight} given by 
\begin{equation} \label{eqw}
w_{M\kappa} = 1/ |F_{\kappa}|
\end{equation}
for any $M$ in $F_{\kappa}$, and the \emph{proportional-to-probability weight (PPW)} given by 
\begin{equation} \label{ppw}
w_{M\kappa} = \pi_M / \pi_{(\kappa)} \qquad\text{and}\qquad \pi_{(\kappa)} = \sum_{M\in F_{\kappa}} \pi_M 
\end{equation}
Notice that the PPW \eqref{ppw} is only possible if $F_{\kappa} \subseteq \mathcal{C}_s$. 
 
Given the AS3 $s_{\kappa}$ is of the order $|s_{\kappa}| = q+1$ and $|s|=n$, there are at most $n-q$ possible estimators based on $(X_t, ..., X_{t+q})$ by \eqref{LRW-IWE}, denoted by $\hat{\theta}_t$ for $t = 1, ..., n-q$. Let $\mathbb{I}_t = 1$ indicate that $(X_t, ..., X_{t+q})$ can lead to the observation of at least one motif in $\Omega$, and $\mathbb{I}_t =0$ otherwise. An estimator of $\theta$ combining all the $\hat{\theta}_t$ is given by
\begin{equation} \label{LRW-theta}
\hat{\theta} = \Big( \sum_{t=1}^{n-q} \mathbb{I}_t \hat{\theta}_t \Big) / \Big( \sum_{t=1}^{n-q} \mathbb{I}_t \Big) ~.
\end{equation}
 
Generally, the estimators \eqref{LRW-IWE} and \eqref{LRW-theta} can only be calculating using $\hat{\pi}_M$ instead of $\pi_M$, which requires an estimate of $R$. However, it is also possible to use $\pi_M$ directly to estimate graph parameters regardless the proportionality constant $2R + rN$. For instance, let $\mu = \theta/N_{\Omega}$, where $N_{\Omega} = |\Omega|$. A generalised ratio estimator is given by 
\[
\hat{\mu} = \hat{\theta}/\hat{N}_{\Omega}
\] 
where $\hat{N}_{\Omega}$ is also given by \eqref{LRW-IWE} and \eqref{LRW-theta} but with $y_{\kappa} \equiv 1$. In the special case of $\Omega = U$, this reduces to the GR estimator of population mean $\sum_{i\in U} y_i/N
$ (Thompson, 2006), as a 1st-order graph parameter.
Generally one can estimate any $\mu$ that is a function of graph totals, as long as the plug-in $\hat{\mu}$ is invariant when each total involved is replaced by its estimator \eqref{LRW-theta}.

\section{Illustration} \label{illustration}

Let $G = (U, A)$ be simple and undirected with $N = |U| = 100$. Let $y=1$ be the value associated with the first 20 nodes $i=1, ..., 20$, to be referred to as the cases, and let $y=0$ be the value for the rest 80 nodes or the noncases. 
The edges $(ij)$ are generated randomly, with different probabilities according to $y_i + y_j=2$, 1 or 0. The resulting graph has 299 edges, $R =|A| = 299$; the cases have an average degree 13.5, and the noncases have an average degree 4.1. This valued graph $G$ with a mild core-periphery structure is fixed for the illustrations below.

\subsection{Convergence to equilibrium} 

Let $p_{t,i} = \mbox{Pr}(X_t = i)$, for $i\in U$. We have $p_{t,i} \rightarrow \pi_i \propto d_i +r$ by LRW, as $t\rightarrow \infty$. How quickly a walk reaches equilibrium is affected by the selection of the initial state $X_0$. In the extreme case the walk is at equilibrium from the beginning if $\mbox{Pr}(X_0 = i) = \pi_i$ for $i\in U$. To explore the speed of convergence, consider two other choices $p_{0,i} = \mbox{Pr}(X_0 = i) \equiv 1/N$ or $p_{0,1}  = 1$. To track the convergence empirically, we use $B$ independent simulations of LRW to estimate 
$E(Y_t) = \sum_{i\in U} p_{t,i} y_i$, which is targeted at the equilibrium expectation $E(Y_{\infty}) = \sum_{i\in U} \pi_i y_i$.

\begin{table}[ht]
\centering
\caption{$E(Y_t)$ by $t$, $r$ and initial $X_0$, $10^5$ simulations of LRW with $w=1$}
\begin{tabular}{l | cccc | cccc} \toprule
& \multicolumn{4}{|c|}{$r=1,~ E(Y_{\infty}) = 0.415$} & \multicolumn{4}{c}{$r=0.1,~ E(Y_{\infty}) =0.447$} \\ \cline{2-9}
Initiation & $t=1$ & $t=4$ & $t=8$ & $t=16$ & $t=1$ & $t=4$ & $t=8$ & $t=16$ \\ \hline
$p_{0,i} = \pi_i$ & 0.420 & 0.419 & 0.421 & 0.414 & 0.449 & 0.451 & 0.450 & 0.448 \\
$p_{0,i} = 1/N$ & 0.341 & 0.394 & 0.408 & 0.413 & 0.355 & 0.412 & 0.433 & 0.449 \\
$p_{0,1} = 1$ & 0.714 & 0.481 & 0.433 & 0.413 & 0.741 & 0.529 & 0.471 & 0.448 \\ \bottomrule
\end{tabular} \label{tab:cnv}
\end{table}

Table \ref{tab:cnv} shows the results for $t = 1, 4, 8, 16$ and $r = 1, 0.1$, each based on $10^5$ simulations of LRW with $w=1$. The equilibrium expectation $E(Y_{\infty})$ varies with $r$. Of the two choices here, the value $r=1$ yields the degree+1 walk, whereas the value $r=0.1$ tunes the walk closer to pure RW. It can be seen that the walk stays at equilibrium if $p_{0,i} = \pi_i$. Under the current set-up, convergence to equilibrium can apparently be expected at $t=16$, whether the initial $X_0$ is selected completely randomly from $U$ given $p_{0,i} = 1/N$, or fixed at $i=1$ given $p_{0,1} = 1$. Neither does the speed of convergence vary much for the values of $r$ here.

\subsection{Estimation of case prevalence} 

Consider the 1st-order parameter $\mu = \sum_{i\in U} y_i/N$. Let $s = \{  X_0, ..., X_T\}$, where $X_0$ is simply drawn with $\pi_i$ now that convergence to equilibrium is quick and what matters for the estimation is mainly the size of the seed sample size itself. Let $\psi$ be the \emph{traverse} of the walk, given as the ratio between the number of distinct nodes visited by the walk and $N$, which indicates how extensively the walk has travelled through $G$.

Given $(T,r,w)$, generate LRW independently $B$ times, each resulting in a replicate of $\hat{\mu}$, the mean of which is an estimate of $E(\hat{\mu})$ by LRW sampling at equilibrium, denoted by Mean$(\hat{\mu})$, and the square root of their empirical variance--SD$(\hat{\mu})$--is an estimate of the standard error of $\hat{\mu}$.

\begin{table}[ht]
\centering
\caption{Estimation of case prevalence $\mu = 0.2$, $10^3$ simulations}
\begin{tabular}{c|ccc|ccc} \toprule
& \multicolumn{3}{|c|}{$r=0.1$} & \multicolumn{3}{c}{$r=6$} \\ \cline{2-7}
$T=50$ & Mean$(\hat{\mu})$ & SD$(\hat{\mu})$ & Mean($\psi$) & Mean$(\hat{\mu})$ & SD$(\hat{\mu})$ & Mean($\psi$) \\ \hline
$w=1$ & 0.211 & 0.093 & 0.321 & 0.203 & 0.070 & 0.383  \\ \hline
$w=0.01$ & 0.210 & 0.079 & 0.367 & 0.203 & 0.067 & 0.395   \\ \midrule
& \multicolumn{3}{|c|}{$r=0.1$} & \multicolumn{3}{c}{$r=6$} \\ \cline{2-7}
$T=100$ & Mean$(\hat{\mu})$ & SD$(\hat{\mu})$ & Mean($\psi$) & Mean$(\hat{\mu})$ & SD$(\hat{\mu})$ & Mean($\psi$) \\ \hline
$w=1$ & 0.204 & \textcolor{purple}{0.066} & 0.499 & 0.200 & 0.049 & 0.607  \\ \hline
$w=0.01$ & 0.203 & \textcolor{purple}{0.054} & 0.559 & 0.201 & 0.048 & 0.617 \\ \bottomrule
\end{tabular} \label{tab:node}
\end{table}

Table \ref{tab:node} gives the results for $T=50$ or 100, $r=6$ or 0.1, and $w=1$ or 0.01, each based on $10^3$ simulations. Since the case nodes have larger degrees than the noncase nodes in $G$, the stationary probability $\pi_i$ depends on $y_i$ and LRW sampling is informative in this sense. Nevertheless, this is not an issue for design-based estimation of graph parameters, and the consistency of $\hat{\mu}$ is already evident at $T = 50$.  

Mean($\psi$) shows that in the current setting a TRW (i.e. $w=1$) of length $T=50$ is expected to visit only about a third of all the nodes in $G$, and one of length $T=100$ can reach about half of the nodes. How quickly TRW traverses the graph depends on the chance of visiting isolated nodes by random jumps as well as the chance of backtracking to the previous adjacent node, the probabilities of both are reduced given small $r$ and $w$. By reducing the chance of backtracking, LRW with $w=0.01$ can speed up the traverse. However, increasing the chance of random jumps, e.g. $r=6$ instead of $r=0.1$, can speed up the traverse even more. Note that since the average degree $2R/N$ is about 6 in $G$ here, setting $r=6$ in \eqref{LRW} makes a random jump on average at least as probable as an adjacent move at each time step. 

Given that we are estimating a 1st-order parameter here, every distinct node encountered contributes more effectively to the estimation than revisiting a node. As can be seen from SD($\hat{\mu}$), increasing $r$ from 0.1 to 6 is more beneficial to the sampling efficiency than reducing $w$ from 1 to 0.01. Nevertheless, reducing $w$ can greatly reduce the sampling variance of TRW at $r=0.1$, e.g. by about 1/3 given $T=100$ as highlighted in Table \ref{tab:node}.

\subsection{Estimation of graph size} 
  
Consider the three estimators \eqref{CR}, \eqref{GR} and \eqref{Hajek} of the 2nd-order graph total $R$. We use simply $n_x = n_y = n$ for each pair of independent $\{ X_t \}$ and $\{ Y_t \}$. Table \ref{tab:size} gives the results for $n=50$ or 100, $r=6$ or 0.1, and $w=1$ or 0.01, each based on $10^4$ simulations, where an estimate of the expectation of each $\hat{R}$ is given together with its simulation error in the parentheses.
  
\begin{table}[ht]
\centering
\caption{Estimation of graph size $R= 299$ (SD in parentheses), $10^4$ simulations}
\begin{tabular}{c|ccc|ccc} \toprule
& \multicolumn{3}{|c|}{$r=0.1$} & \multicolumn{3}{c}{$r=6$} \\ \cline{2-7}
$n=50$ & $\hat{R}_1$ & $\hat{R}_2$ & $\hat{R}_3$ & $\hat{R}_1$ & $\hat{R}_2$ & $\hat{R}_3$ \\ \hline
$w=1$ & 344 (1.64) & 306 (0.41) & 344 (\textcolor{purple}{1.61}) & 333 (1.61) & 300 (0.26) & 316 (0.81) \\ \hline
$w=0.01$ & 319 (0.93) & 304 (0.36) & 319 (\textcolor{purple}{0.91}) & 329 (1.43) & 300 (0.26) & 314 (0.72) \\ \midrule
& \multicolumn{3}{|c|}{$r=0.1$} & \multicolumn{3}{c}{$r=6$} \\ \cline{2-7}
$n=100$ & $\hat{R}_1$ & $\hat{R}_2$ & $\hat{R}_3$ & $\hat{R}_1$ & $\hat{R}_2$ & $\hat{R}_3$ \\ \hline
$w=1$ & 311 (0.61) & 303 (0.31) & 311 (0.60) & 307 (0.70) & 299 (0.19) & 303 (0.36)  \\ \hline
$w=0.01$ & 305 (0.42) & 302 (0.27) & 305 (0.41) & 306 (0.65) & 299 (0.18) & 302 (0.33) \\ \bottomrule
\end{tabular} \label{tab:size}
\end{table}

The traverse for given $(T,r,w)$ is already reported earlier. At $n=50$, the average number of collisions between $\{ X_t \}$ and $\{ Y_t \}$ is about 4 given $r=0.1$, which is about halved given $r=6$; at $n=100$, it is about 16 given $r=0.1$, and about halved given $r=6$. Reducing $w=1$ to 0.01 improves the CR estimator $\hat{R}_1$ more than increasing $r=0.1$ to 6. The improvement is about the same with either device for the GR-CR estimator $\hat{R}_3$. The GR estimator is the best of the three here, because the number of collisions is still quite low whether $n=50$ or 100. 

Regardless the estimator, LRW sampling with $w=0.01$ is more efficient than TRW sampling (i.e. $w=1$). The gain is largest for the two CR-related estimators, which is most useful when the TRW sampling variance is largest, e.g. as highlighted in Table \ref{tab:size}.

\subsection{Estimation of a 3rd-order graph total and a related parameter} 

Let $\theta$ be the total number of triangles in $G$, and $\theta_1$ that of the case triangles, i.e. $\prod_{i\in M} y_i = 1$ if $[M]$ is triangle. The 3rd-order graph parameter $\mu = \theta_1/\theta$ is a measure of the transitivity among cases compared to the overall transitivity in $G$. We have $(\theta, \theta_1, \mu) = (170, 140, 0.824)$ here.

Let $\Omega$ contain all the triangles in $G$. Given $(X_t, X_{t+1}) = (i,j)$ of two adjacent nodes under LRW, one observes all the triangles involving $i$ and $j$. For instance, both the triangles of $\{ 1,2,3\}$ and $\{ 1,2,\diamond\}$ in Figure \ref{fig:LRW} are observed from $(X_t, X_{t+1}) = (1,2)$, such that for \eqref{LRW-IWE} we have $\delta_M = 1$ if $M = (1,2)$ and $I_{\kappa}(M) =1$ if $\kappa$ is either of these two triangles. 

Whereas $\mu$ can be estimated using $\pi_M$ directly, $\hat{\pi}_M = \pi_M(\hat{R})$ is needed for $\theta$ or $\theta_1$.
Given any $(X_t, X_{t+1}) = (i,j)$ as the AS3 of a triangle $\kappa$, the ES3s are the six possible adjacent moves along the triangle. The incidence weight is the same, i.e. $w_{M\kappa} \equiv 1/6$, either by \eqref{eqw} or \eqref{ppw}, because the S3P \eqref{piM} of LRW \eqref{LRW} is a constant of $M =(i,j)$, i.e.
\[
\pi_M = \pi_i p_{ij} \propto 1 + r/N 
\] 

\begin{table}[ht]
\centering
\caption{Estimation of $\mu = \theta_1/\theta = 0.824$ ($B=10^3$) and $\theta = 170$ ($B =10^4$)}
\begin{tabular}{c|ccc|ccc} \toprule
 & \multicolumn{3}{|c|}{$r=0.1$} & \multicolumn{3}{c}{$r=6$} \\ \cline{2-7}
$T=50$ & Mean$(\hat{\mu})$ & SD$(\hat{\mu})$ & $\hat{\theta}$ & Mean$(\hat{\mu})$ & SD$(\hat{\mu})$ & $\hat{\theta}$ \\ \hline
$w=1$ & 0.796 & 0.108 & 185 (0.91) & 0.796 & 0.125 & 174 (0.84) \\ \hline
$w=0.01$ & 0.797 & 0.101 & 180 (0.78) & 0.799 & 0.124 & 175 (0.82) \\ \midrule
& \multicolumn{3}{|c|}{$r=0.1$} & \multicolumn{3}{c}{$r=6$} \\ \cline{2-7}
$T=100$ & Mean$(\hat{\mu})$ & SD$(\hat{\mu})$ & $\hat{\theta}$ & Mean$(\hat{\mu})$ & SD$(\hat{\mu})$ & $\hat{\theta}$ \\ \hline
$w=1$ & 0.825 & \textcolor{purple}{0.082} & 177 (0.63) & 0.815 & 0.101 & 172 (0.59) \\ \hline
$w=0.01$ & 0.826 & \textcolor{purple}{0.073} & 175 (0.55) & 0.811 & 0.096 & 173 (0.58) \\ \bottomrule
\end{tabular} \label{tab:triangle}
\end{table}

Table \ref{tab:triangle} gives the results for $T=50$ or 100, $r=6$ or 0.1, and $w=1$ or 0.01. The number of simulations is $B=10^3$ for $\mu$ and $B=10^4$ for $\theta$. For $\hat{\theta}$ an estimate of Mean($\hat{\theta}$) is given together with its simulation error in the parentheses. The speed of convergence is noticeably quicker for $\hat{\theta}$ with $r=6$ instead of 0.1, but does not differ much for $\hat{\mu}$. On the one hand, intuitively a large chance of random jump is unlikely to be efficient as the order of motif increases; on the other hand, the total estimator $\hat{\theta}$ uses $\pi_M(\hat{R})$ instead of $\pi_M$, and its convergence rate seems to have benefitted from the correlation with $\hat{R}$. Again, the extra flexibility of LRW can be useful in certain situations, e.g. as highlighted in Table \ref{tab:triangle}.

\section{Some remarks on future research} \label{remark}

Theoretical results on the convergence rate only seem to be available for pure random walks in connected graphs; see e.g Boyd et al. (2004), Ben-Hamou et al. (2019). The empirical results  here suggest that longer walks are needed for estimating totals using $\hat{\pi}_M$ than estimating parameters using $\pi_M$ directly, but theoretical results are lacking at the moment.

What is a good choice of $w$ can conceivably vary for different motifs. Using a small $w$ with nearly no backtracking if possible seems to be reasonable in all the illustrations considered here, especially when the walk has a relatively low probability for random jumps. More systematic investigation is needed in this respect.

The LRW can readily be generalised to \emph{$q$-lagged} random walk (LRW$_q$). Given any initial states $(X_0, X_1, ..., X_q)$, let the current and next \emph{tuple} (of ordered states) be $\bs{x}_t = (X_{t-q}, M)$ and $\bs{x}_{t+1} = (M, X_{t+1})$, respectively, where $M = (X_{t-q+1}, ..., X_t)$, with Markovian transition from $\bs{x}_t$ to $\bs{x}_{t+1}$ (and $X_{t+1}$).
It is intriguing to find the stationary distribution of $X_t$ by LRW$_q$.

The choice of incidence weights $w_{M\kappa}$ over $F_{\kappa}$ in nontrivial generally, although the two basic choices \eqref{eqw} or \eqref{ppw} happen to coincide for the triangle motif. For instance, let $M =\{i,j,g,h\}$ form a 4-cycle, which can be observed based on AS3 consisting of two successive adjacent moves, say, $(X_t, X_{t+1}, X_{t+2}) = (i, j, g)$. The corresponding S3P \eqref{piM} is 
\[
\pi_M = \pi_i p_{ij} p_{jg} = \frac{1}{d_j +r} \big( 1 + \frac{r}{N} \big)^2 
\]
such that the PPW \eqref{ppw} would differ to the multiplicity weight \eqref{eqw}. Since the PPW requires $F_{\kappa} \subseteq \mathcal{C}_s$ but not the equal weights, the latter can be applied to a larger number of sample motifs. In practice one can calculate several estimators using different incidence weights and choose among them according to a given situation.

At any moment the traverse of a walk is either known given $N$ or can be estimated together with $\hat{R}_3$. A prediction form of \eqref{LRW-theta} would be helpful, which takes into account the fact that uncertainty only arises due to the part of graph that is yet unobserved.

\end{document}